\documentclass[12pt]{amsart}
\usepackage{amssymb, amsmath}
\usepackage[utf8]{inputenc}

\begin{document}
\newcommand{\xc}{\mathcal{X}}\newcommand{\yc}{\mathcal{Y}}
\newcommand{\ch}{\mathcal{H}}\newcommand{\cs}{\mathcal{S}}
\newcommand{\cv}{\mathcal{V}}\newcommand{\cd}{\mathcal{D}}
\newcommand{\cf}{\mathcal{F}}
\newcommand{\mt}{\mbox{}^t}
\newcommand{\pb}{\mathbb{P}}\newcommand{\cb}{\mathbb{C}}\newcommand{\qb}{\mathbb{Q}}
\newcommand{\zb}{\mathbb{Z}}\newcommand{\rb}{\mathbb{R}}\newcommand{\nb}{\mathbb{N}}
\newcommand{\oc}{\mathcal{O}}\newcommand{\lc}{\mathcal{L}}
\newcommand{\ep}{\varepsilon}

\newcommand{\Pic}{\mathop{Pic}}

\theoremstyle{plain}
  \newtheorem{thm}{Theorem}
  \newtheorem*{main}{Main Theorem}
  \newtheorem{defthm}[thm]{Definition-Theorem}
  \newtheorem{prop}{Proposition}[section]
  \newtheorem{lem}[prop]{Lemma}
  \newtheorem{cor}[thm]{Corollary}
  \newtheorem*{conj}{Conjecture}
\newtheorem{Def}{Definition}[section]
\numberwithin{equation}{section}

\title{Families of hypersurfaces of large degree}
\author{Christophe Mourougane}

\address{Christophe Mourougane\\Institut de Recherche Mathématiques de Rennes (IRMAR)\\
Campus de Beaulieu\\ 35402 Rennes.}
\email{christophe.mourougane@univ-rennes1.fr}
\date{\today}\maketitle
\begin{abstract}
We show that general moving enough families of high degree hypersurfaces in $\pb^{n+1}$ 
do not have a dominant set of sections.
\end{abstract}

\section{Introduction}
In ~\cite{grauert}, Grauert solved Mordell conjecture for curves over function fields.
Lang generalised this statement in~\cite{lang}.
\begin{conj}[Lang’s conjecture over function fields]
 Let $\pi : X\to Y$ be a projective surjective morphism of complex algebraic manifolds, whose generic fibre is of general type. If $f$ is not
birationally trivial, then there is a proper subscheme of $X$  that contains the image of all sections of $\pi$.
\end{conj}

Grauert's proof can be red as a construction of first order differential equations fulfilled by all but finite number of sections of the family. First order differential equations are also enough to deal with families of manifolds with ample cotangent bundles~(\cite{noguchi2}\cite{moriwaki}).
We implement this idea in higher dimensions with higher order differential equations, in a case where the positivity assumption is made only for the canonical bundle,  to prove the
\begin{main}
 \label{theomain}  
For general moving enough families of  high enough degree hypersurfaces in $\pb^{n+1}$, there is a proper algebraic subset of the total space that contains the image of all 
 sections. 
\end{main}
General families are in particular required  to be non-birationally isotrivial.
By definition, ``moving enough families'' are those families parametrised by a curve $B$ with variation given by a line bundle on $B$ whose degree is large and $(n+1)$-times the degree of a rational function on $B$.

We point out that Noguchi~\cite{noguchi} gave a proof of Lang's conjecture when the members of the family are assumed to be hyperbolic and the smooth part is assumed to be hyperbolically embedded. Eventhough Kobayashi conjectured that a generic hypersurface of large degree in $\pb^{n+1}$ is hyperbolic, our proof does not rely on properties of families hyperbolic manifolds, like normality. We benefit however from the recent works dealing with Kobayashi conjecture, especially from~\cite{demailly}.

The first part of the work describes general tools for dealing with higher order jets.
The second part is devoted to the proof of the
\begin{thm}\label{theo1}
All sections of general moving enough families of  high enough degree hypersurfaces in $\pb^{n+1}$ fulfil a differential equation of order $n+1$. 
\end{thm}
Then, adapting general techniques in universal families originating in the work of Clemens~\cite{clemens}, Voisin \cite{voisin} and Siu~\cite{siu}, we obtain the main theorem in the third part.

\thanks{
I thank Claire Voisin for the nice idea she gave to me for computing nef cones. I~discussed the subject of this paper with many people over more than four years.
I would like to thank them all, in a single sentence.}

\section{Jet spaces for sections}
We consider a smooth proper connected family $\pi~:~\xc\to B$ of $n$-dimensional manifolds parametrised by a connected curve $B$, and we intend to construct the jet spaces for sections of~$\pi$.

\subsection{Jets of order one}
We follow the ideas of Grauert~\cite{grauert}.

Consider a section $s~:~B_\rho\to\xc$ of the pull-back family $\pi_\rho~:~\rho^\star\xc\to B_\rho$, where $\rho~:~B_\rho\to B$ is a finite morphism of curves. The map 
$\mt ds~:~ s^\star \Omega_\xc \to \Omega_{B_\rho}$ satisfies $\mt ds\circ s^\star\mt d\pi_\rho=Id_{\Omega_{B_\rho}}$, is hence surjective and provides a rank one quotient of $s^\star \Omega_\xc$.

The corresponding curve $s_1~:~B_\rho\to \xc_1$ inside the bundle 
$\pi_{0,1}~:~\xc_1:=\pb (\Omega_\xc)\to \xc$
of rank one quotients of $\Omega_\xc$ lifts $s$ (i.e. $\pi_{0,1}\circ s_1=s$),
is therefore a section of $\pi_1~:~\xc_1\to B_\rho$ and avoids the divisor $\cd_1:=\pb (\Omega_{\xc/B})$ of vertical differentials.
The latter is the divisor of the section of 
$\pi^\star T_{B}\otimes\oc_{\Omega_\xc}(1)$ 
given by $\mt d\pi~:~\pi^\star\Omega_{B}\to{\Omega_\xc}$. 
We have to study the positivity properties of this line bundle, which transfer into mobility properties of the forbidden divisor $\cd_1$.

\subsection{Second order jets}

As in the preceding section, the curve $s_1~:~B_\rho\to \xc_1$ lifts to a curve inside the bundle of rank one quotients of $\Omega_{\xc_1}$. 
More precisely, the rank one quotient $\mt ds_1~:~ s_1^\star \Omega_{\xc_1} \to\Omega_{B_\rho}$
fulfils the relation $$\mt ds_1\circ s_1^\star \mt d\pi_{0,1}=\mt ds.$$
The map $\mt ds_1$ at the point $[\mt ds]$ of $\xc_1$ vanishes on the image by $\mt d\pi_{0,1}$ of forms in the kernel of the tautological quotient $\mt ds$. 
In other words, $\mt ds_1$ is a rank one quotient of the quotient $\cf_1$ of $\Omega _{\xc_1}$ defined by the following diagram on $\xc_1$. 
$$
\begin{array}{ccccccc} 
&  0            &            &    0        &    &                    &\\
& \downarrow    &            & \downarrow  &    &                    &\\
& S             &      =  &  S            &    &                     &\\
& \downarrow    &         & \mbox{\ \ \ \ \ \ \ \ \ }\downarrow   \mt d\pi_{0,1} 
                                          &    &                     &\\
0\to & \pi_{0,1}^\star\Omega_\xc &\xrightarrow{\mt d\pi_{0,1}}&\Omega _{\xc_1} &\to &\Omega _{\xc_1/ \xc}&\to 0\\
& \downarrow    &          & \downarrow    &    &       ||            &\\
0\to & \oc_{\xc_1}(1) &\xrightarrow{\;\;}& \cf_1&\to &\Omega _{\xc_1/ \xc}&\to 0\\
&  \downarrow   &          & \downarrow     &    &       \downarrow    &\\
&  0            &          &  0            &  &   0                    & \\
\end{array}
$$

Define the second order jet space to be 
$\pi_{1,2}~:~\xc_2:=\pb (\cf_1)\to \xc_1$. 
As in the formalism of Arrondo, Sols and Speiser~\cite{ass}, we need to keep track of the injective map 
$a_2~:~\xc_2\to \pb (\Omega _{\xc_1})$ given by the quotient 
$\Omega _{\xc_1}\to\cf_1$.

We hence get a map $s_2~:~B_\rho\to \xc_2$ defined by the quotient 
$\mt ds_1~:~s_1^\star{\cf_1}\to \Omega_{B_\rho}$. 
Note that $\pi_{1,2}$ is the restriction to 
$\pb (\cf_1)\subset\pb (\Omega _{\xc_1})$ 
of the map defined by the quotient 
$\pi_1^\star\Omega_\xc\xrightarrow{\mt d\pi_{0,1}}\Omega _{\xc_1}$ so that the relation $\mt ds_1\circ s_1^\star \mt d\pi_{0,1}=\mt ds$ is rephrased in the fact that the map $s_2$ is a lifting of $s_1$ (i.e. $\pi_{1,2}\circ s_2=s_1$). 

The map $\mt d\pi_{0,1} : \oc_{\Omega_\xc}(1)\to \cf_1$ gives rise to a section of $\pi_1^\star\oc_{\Omega_\xc}(-1)\otimes \oc_{\cf_1}(1)$ whose divisor 
$\cd_2:=\pb(\Omega _{\xc_1/ \xc})\subset \xc_2$ is not hit by the curve $s_2$ for 
$\mt ds_1\circ s_1^\star \mt d\pi_{0,1}$ vanishes nowhere.

\subsection{Higher order jets}

This scheme inductively leads to the construction of the $k^{th}$-order jet spaces 
$\pi_{k-1,k}~:~\xc_{k}\to \xc_{k-1}$, together with a map 
$a_k~:~\xc_{k}\to \pb (\Omega _{\xc_{k-1}})$
that completes the commutative diagram.
$$
\begin{array}{cccc}
 \xc_{k}=\pb(\cf_{k-1})&\xrightarrow{a_{k}}& \pb (\Omega _{\xc_{k-1}})\\
       \pi_{k-1,k}\searrow &   & \swarrow p_{k-1} & \\
     &   \xc_{k-1}          &              &
\end{array}
$$
Note that $a_{k}^\star\oc_{\Omega_{\xc_{k-1}}}(1)=\oc_{\xc_k}(1)$.
 The bundle $\cf_k$ on $\xc_k$ is the quotient of $\Omega_{\xc_k}$ defined by 
$$
\begin{array}{ccccccc}
&  0         &              &    0      &    &                &\\
& \downarrow &              & \downarrow&    &                &\\
& S_k       &      =        &  S_k      &    &                &\\
& \downarrow&               & \mbox{\ \ \ \ \ \ }\downarrow \mt d\pi_{k-1,k}   &    &                     &\\
0\to & \pi_{k-1,k}^\star\Omega_{\xc_{k-1}}&\xrightarrow{\mt d\pi_{k-1,k}}&\Omega _{\xc_k} &\to &\Omega _{\xc_k / \xc_{k-1}}&\to 0\\
& \downarrow&               & \downarrow &    &        ||       &\\
 0\to & \oc_{\xc_k}(1) &\xrightarrow{\mbox{\;\;}} & \cf_k &
 \to &\Omega _{\xc_k / \xc_{k-1}}&\to 0\\
&  \downarrow &              & \downarrow&  & \downarrow      &\\
&  0          &              &  0        &  &  0              & \\
\end{array}
$$

The $(k+1)^{th}$-order jet space is 
$\pi_{k,k+1}~:~\xc_{k+1}:=\pb (\cf_k)\to\xc_k$ and the map $a_{k+1}~:~\xc_{k+1}\to \pb (\Omega _{\xc_k})$ is the injective map 
associated with the quotient $\Omega_{\xc_k}\to\cf_k$.
Note that the relative dimension of $\pi_{k+1,k}$ is equal to that of $\pi_{k-1,k}$ that is $n$. Therefore $$\dim \xc_k=(k+1)n+1.$$

Now, given a section $s~:~B_\rho\to\xc$ of the pull-back family $\pi_\rho~:~\rho^\star\xc\to B_\rho$,
by a finite morphism of curves $r~:~B_\rho\to B$, assuming that we have constructed the various lifts 
$s_i~:~B_\rho\to \xc_i$, up to the level $k$, we get the $(k+1)^{th}$-order jet $s_{k+1}~:~B_\rho\to\xc_{k+1}$ by considering the surjective map $\mt ds_{k}~:~s_k^\star\cf_k\to \Omega_{B_\rho}$ built from the relation 
$\mt ds_k\circ s_k^\star \mt d\pi_{k-1,k} =\mt ds_{k-1}$. 

Recall that the tautological quotient bundle $\oc_{\xc_{k+1}}(1)$ 
pulls-back to $B_\rho$ via $s_{k+1}$ into the considered quotient $\Omega_{B_\rho}$
$$s_{k+1}^\star\oc_{\xc_{k+1}}(1)=\Omega_{B_\rho}.$$

The map $\mt d\pi_{k-1,k} : \oc_{\xc_k}(1)\to\cf_k$ gives rise to a divisor 
$\cd_{k+1}=\pb(\Omega_{\xc_{k}/\xc_{k-1}})$ on $\xc_{k+1}$ in the linear system
 $|\pi_{k+1,k}^\star\oc_{\xc_{k}}(-1)\otimes \oc_{\xc_{k+1}}(1)|$ 
that the curve $s_{k+1}$ avoids. 

\subsection{Description in coordinates}
Choose a local coordinate $t$ on $B$ and a adapted system of local coordinates on $\xc$, 
$(t,z_1,z_2,\cdots ,z_n)$ such that the map $\pi$ is given
 by $(t,z_1,z_2,\cdots ,z_n)\mapsto t$.
The set of vectors 
$\frac{\partial}{\partial t},\frac{\partial}{\partial z_1},
\frac{\partial}{\partial z_2},\cdots ,\frac{\partial}{\partial z_n}$ provides us 
with a local frame for $T_\xc$. 
This defines relative homogeneous coordinates $[T_1:A_1:A_2:\cdots:A_n]$ on $\xc_1$.

A section $s$ of $\pi$ locally written as 
$t\mapsto(t,z_1(t),z_2(t),\cdots ,z_n(t))$ is differentiated in
$$ds~:~\frac{\partial}{\partial t}\mapsto 
\frac{\partial}{\partial t}+z'_1(t)\frac{\partial}{\partial z_1}+
z'_2(t)\frac{\partial}{\partial z_2}+\cdots 
+z'_n(t)\frac{\partial}{\partial z_n}.$$
The first order jet of the curve $s$ is therefore locally written as 
$s_1~:~B\mapsto\xc_1=P(T_\xc)$,
$$s_1~:~t\mapsto (t,z_1(t),z_2(t),\cdots,z_n(t),[1:z'_1(t):z'_2(t):\cdots:z'_n(t)]).$$
It does not meet the divisor $\cd_1:=P(T_{\xc/B})$ locally given by $T_1=0$.

Outside this divisor, we get relative affine coordinates 
$a_1:=A_1/T_1,a_2:=A_2/T_1, \cdots ,  a_n:=A_n/T_1$.
Note that for the section $s_1$ we infer that $a_j(t)=z'_j(t)$.
The set of vectors 
$$\frac{\partial}{\partial t},\frac{\partial}{\partial z_1},
\frac{\partial}{\partial z_2},  \cdots ,
\frac{\partial}{\partial z_n}, \frac{\partial}{\partial a_1},
\frac{\partial}{\partial a_2},
\cdots ,\frac{\partial}{\partial a_n}$$ provides us with a local frame for $T_{\xc_1}$.
The bundle $\cf_1^\star$ is defined to be 
$$\cf_1^\star:=\{(t,z,[A],v)\in T_{\xc_1}\, /\, d\pi_{0,1}(v)\in [A]\subset T_\xc\}.$$
It has a local frame built with
$$\frac{\partial}{\partial t}+a_1\frac{\partial}{\partial z_1}+
a_2\frac{\partial}{\partial z_2}+\cdots +a_n\frac{\partial}{\partial z_n}
\in \oc_{\xc_1}(-1)$$
and 
$$\frac{\partial}{\partial a_i}\in T_{\xc_1/\xc} , \ \  1\leq i\leq n.$$
This defines relative homogeneous coordinates $[T_2:B_1:B_2:\cdots:B_n]$ on $\xc_2$.

The section $s_1~:~B\to \xc_1-\cd_1$ of $\pi_{1}$ is differentiated in
\begin{eqnarray*}
\lefteqn{ds_1~:~\frac{\partial}{\partial t}\mapsto }\\
&\frac{\partial}{\partial t}+z'_1(t)\frac{\partial}{\partial z_1}
+z'_2(t)\frac{\partial}{\partial z_2}+\cdots +z'_n(t)\frac{\partial}{\partial z_n}
+z''_1(t)\frac{\partial}{\partial a_1}+z''_2(t)\frac{\partial}{\partial a_2} +\cdots+z''_n(t)\frac{\partial}{\partial a_n}\\
=
&\left(\frac{\partial}{\partial t}+a_1(t)\frac{\partial}{\partial z_1}
+a_2(t)\frac{\partial}{\partial z_2}+\cdots 
+a_n(t)\frac{\partial}{\partial z_n}\right)
+z''_1(t)\frac{\partial}{\partial a_1}+z''_2(t)\frac{\partial}{\partial a_2} +\cdots+z''_n(t)\frac{\partial}{\partial a_n}.
\end{eqnarray*}
The second order jet $s_2~:~B\to\xc_2$ is locally written
as
$$s_2~:~t\mapsto (t,z_1(t),z_2(t),\cdots,z_n(t), [1:z'_1(t):z'_2(t):\cdots:z'_n(t)], [1:z''_1(t):z''_2(t):\cdots:z''_n(t)]).$$
It does not meet the divisor $\cd_2:=P(T_{\xc_2/\xc_1})$ locally given by $T_2=0$.

Coordinates in higher order jet spaces are defined similarly.

\section{Schwarz lemma and holomorphic Morse inequalities}
This section is devoted to the tools needed to prove theorem~\ref{theo1}.
Consider the line bundles on $\xc_{k}$ defined by
$$\oc_{\xc_{k}}(\underline{m})
:=\pi^\star_{1,k}\oc_{\xc_{1}}(m_1)\otimes\pi^\star_{2,k}\oc_{\xc_{2}}(m_2)
\otimes\cdots\otimes\oc_{\xc_{k}}(m_k)$$ 
and  $$\oc_{\xc_{k}}(\underline{M\cd})
:=\pi^\star_{1,k}\oc_{\xc_{1}}(M_1\cd_1)
\otimes\pi^\star_{2,k}\oc_{\xc_{2}}(M_2\cd_2)
\otimes\cdots\otimes\oc_{\xc_{k}}(M_k\cd_k).$$
Define  \begin{eqnarray*}
\chi_\rho&:=&\int_{B_\rho} s_1(\Omega_{B_\rho})=-\int_{B_\rho}c_1(T_{B_\rho})=-2\int_{B_\rho} Todd(T_{B_\rho})
=-2\chi (B_\rho)=2g(B_\rho)-2\\&\geq& (\deg\rho)(2g(B)-2)\geq 0.\end{eqnarray*}

Consider a section $\sigma$ of the line bundle $\oc_{\xc_{k}}(\underline{m})\otimes\oc_{\xc_{k}}(\underline{M\cd})
\otimes\pi_{k}^\star \lambda^{-1}$. 
Pull it back to $B_\rho$ via the map $s_k~:~B_\rho\to\xc_k$ into a section $s_k^\star \sigma$ of the line bundle $\Omega_{B_\rho}^{\otimes |\underline{m}|}\otimes \rho^\star \lambda^{-1}$. If the latter bundle has an ample dual bundle 
(i.e. if $\displaystyle\deg \lambda>\frac{\chi_\rho}{\deg\rho}\mid\underline{m}\mid$), 
then the section $s_k^\star \sigma$ has to vanish. This gives 
\begin{lem}[Schwarz lemma]
If  a line bundle $\lambda$ on $B$ has degree $\deg\lambda$ greater than $\frac{\chi_\rho}{\deg\rho}\mid\underline{m}\mid$, 
then for every section $s$ of $\rho^\star\xc\to B_\rho$ and every section $\sigma$ 
of the line bundle
$\oc_{\xc_{k}}(\underline{m})\otimes\oc_{\xc_{k}}(\underline{M\cd})
\otimes\pi_{k}^\star \lambda^{-1}$ on $\xc_k$, the $k^{th}$-order jet $s_k$ of $s$ 
lies in the zero locus of $\sigma$.
\end{lem}

Note that we considered only those bundles having zero components along the Picard group of $\xc/B$. 
For example, in the family of hypersurfaces in $\pb^{n+1}$ case,  
bounding the intersection number $s(B)\cdot\oc_{\pb^{n+1}}(1)$, called the height of the section $s$, is a main step in proving Lang's conjecture.
We therefore have to try and produce sections of bundles without any component along the Picard group of $\xc/B$.
We will use the algebraic form of holomorphic Morse inequalities to achieve this.
\begin{prop}[Holomorphic Morse inequalities]
 A multiple $mL$ of a line bundle $L$ on a projective manifold of dimension $D$ that can be written as the difference of two nef line bundles $L=A-B$  is effective and big if furthermore the intersection number $A^D-DA^{D-1}\cdot B$ is positive and $m\geq m_0(c_1(A),c_1(B))$.
\end{prop}
To prove that $m_0$ only depends on  $(c_1(A),c_1(B))$, we just recall that, in the inequality that estimates the alternating sum of dimensions of cohomology groups
$$h^0-h^1(X,L^{\otimes m})\geq \frac{m^D}{D!}( A^D-DA^{D-1})\cdot B+o(m^{D}),$$
the remainder is made of numerical data.

There are three elements to settle to get the proof of theorem~\ref{theo1},
the construction of nef line bundles on jet spaces, 
the inequality $\displaystyle\deg \lambda>\frac{\chi_\rho}{\deg\rho}\mid\underline{m}\mid$
and the positivity of the intersection number $A^D-DA^{D-1}\cdot B$.

Note that we may allow a negative part along the Picard group of $\xc/B$. 
This will give the height estimates.

\section{The nef cones}
We will now restrict to the situation of a family of hypersurfaces in $\pb^{n+1}$ given by a section $s_0$ of an ample line bundle $L_0$ on $B\times\pb^{n+1}$. 
We will assume that the genus of $B$ and the relative dimension $n$ are at least $2$.
$$\begin{array}{ccccc}
   \xc &\xrightarrow{\iota} &B\times\pb^{n+1} &\xrightarrow{pr_2}&\pb^{n+1}\\
{\pi}\downarrow& \swarrow pr_1 &&\\B& && \end{array}$$
The map $pr_2\circ\iota~:~\xc\to \pb^{n+1}$ will be denoted by $R$.
This gives the further sequence on $\xc$
\begin{equation}\label{inj}
 0\to {L_0^\star}_{\mid \xc}\xrightarrow{\mt ds_0} {\Omega_B\oplus\Omega_{\pb^{n+1}}}_{|\xc}
\xrightarrow{\mt d\iota} \Omega_\xc =\cf_0\to 0.
\end{equation}
From Leray-Hirsch theorem, we know that $\Pic (B\times\pb^{n+1})
= pr_1^\star\Pic B\oplus pr_2^\star\Pic \pb^{n+1}$.
In particular, we will write $L_0$ as $pr_1^\star\lambda_0\otimes pr_2^\star\oc_{\pb^{n+1}}(d_0)$.
Note that $\oc_{\pb^{n+1}}(d_0)=(L_0)_{|pr_1^{-1}b}$ is ample ($d_0>0$) 
and $(pr_1)_\star L_0=\lambda_0\otimes S^{d_0}\cb^{n+2}$ is effective ($\deg\lambda_0\geq 0$). 

\subsection{The nef cone of $\xc$}
For the line bundle $L_0$ is assumed to be ample and $\xc$ is of dimension at least $3$, Lefschetz hyperplan theorem reads 
$$\Pic \xc=\iota^\star\Pic (B\times\pb^{n+1})
= \pi^\star\Pic B\oplus R^\star\Pic\pb^{n+1}.$$
For a line bundle $\lambda$ on $B$ and an integer $d$, we will denote by $\oc_\xc(\lambda,d)
=\pi^\star\lambda\otimes R^\star\oc_{\pb^{n+1}}(d)$ 
the restriction to $\xc$ of the line bundle $\lambda\boxtimes\oc_{\pb^{n+1}}(d)=pr_1^\star\lambda
\otimes pr_2^\star \oc_{\pb^{n+1}}(d)$.

The line bundle $\pi^\star\oc_B(b)$, nef but not ample, has its Chern class lying on a vertices of the nef cone of $\xc$.
If the morphism $R~:~\xc\to \pb^{n+1}$ is not finite (e.g. the  section  defining $\xc$ does not involve all the homogeneous coordinates on $\pb^{n+1}$)
then the line bundle $R^\star \oc_{\pb^{n+1}}(d)$ gives the second vertices. This is not the generic case.

The top intersection number of the first Chern class $c_1(\oc_\xc(\lambda,d))\in NS(\xc)$ is given by
\begin{eqnarray*}
\lefteqn{c_1(\oc_\xc(\lambda,d))^{n+1}}&&\\
&=&\iota^\star\left[pr_1^\star c_1(\lambda )+ 
pr_2^\star c_1(\oc_{\pb^{n+1}}(d))\right]^{n+1}\\
&=&\left[ c_1(\lambda_0 )+ c_1(\oc_{\pb^{n+1}}(d_0))\right].
\left[ c_1( \oc_{\pb^{n+1}}(d))^{n+1}+(n+1) c_1(\lambda )
 pr_2^\star c_1(\oc_{\pb^{n+1}}(d))^{n}\right]\\
&=&d^n\left[d \deg(\lambda_0 )+(n+1)d_0 \deg(\lambda )\right].
\end{eqnarray*}
It has to be non-negative on the nef cone.
We hence get in $NS_\rb(\xc)\equiv \rb^2$
\begin{eqnarray*}
\Big\{(l,d)/  
d\geq 0, \quad l\geq 0\Big\}&=&\iota^\star Nef (B\times\pb^{n+1})\\
&&\subset Nef(\xc)
\subset\Big\{(l,d) /   d\geq 0,\quad l\geq -\frac{\deg\lambda_0}{(n+1)d_0}d\Big\}.
\end{eqnarray*}

\subsection{The pseudo-effective cone of $\xc$}
We now compute the pseudo-effective cone, $Eff(\xc)\supset Nef(\xc)$.
Take $\deg\lambda<0$ and $d >0$.
The push-forward by $pr_1$ of the sequence defining the structure sheaf of $\xc$
tensorised by $\lambda\boxtimes\oc_{\pb^{n+1}}(d)$, reads
$$0\to(pr_1)_\star\left(\lambda\otimes\lambda_0^{\star}
\boxtimes\oc_{\pb^{n+1}}(d-d_0)\right)
\to (pr_1)_\star(\lambda\boxtimes\oc_{\pb^{n+1}}(d))
\to \pi_\star\oc_\xc(\lambda,d)\to 0$$
that is
$$0\to\lambda\otimes\lambda_0^{\star}\otimes S^{d-d_0}\cb^{n+2}
\to\lambda\otimes S^{d}\cb^{n+2}
\to\pi_\star\oc_\xc(\lambda,d)\to 0.$$
For $\deg\lambda <0$, the associated long exact sequence gives
\begin{align*}0\to & H^0(\xc,\oc_\xc(\lambda,d))
\to\\ H^1(B,\lambda\otimes\lambda_0^{\star})\otimes S^{d-d_0}\cb^{n+2}
\to H^1(B,\lambda)\otimes S^{d}\cb^{n+2}
\to & H^1(B,\pi_\star\oc_\xc(\lambda,d))\to 0.
\end{align*}
Note that if $\ell$ is large, 
$\mathcal{R}^1 \pi_\star\oc_\xc(\lambda^{\otimes \ell},\ell d)$ vanishes,
so that 
$H^1(B,\pi_\star\oc_\xc(\lambda^{\otimes \ell},\ell d))$
 and $H^1(\xc,\oc_\xc(\lambda^{\otimes \ell},\ell d))$ become isomorphic.
We infer
\begin{align*}
 h^0(\xc,\oc_\xc(\lambda,d)) \geq & h^1(B,\lambda\otimes\lambda_0^{\star})\otimes
S^{d-d_0}\cb^{n+2}-h^1(B,\lambda)\otimes S^{d}\cb^{n+2}\\
\geq & -\chi(B,\lambda\otimes\lambda_0^{\star})\otimes S^{d-d_0}\cb^{n+2}+\chi(B,\lambda)\otimes S^{d}\cb^{n+2}\\
\geq & [\deg\lambda+1-g(B)]\binom{d+n+1}{n+1}\\
&-[\deg\lambda-\deg\lambda_0+1-g(B)]\binom{d-d_0+n+1}{n+1}.
\end{align*}
We find that if $\displaystyle
\deg\lambda> -\frac{\deg\lambda_0}{(n+1)d_0}d$,
for large $\ell$, $H^0(\xc,\oc_\xc(\lambda^{\otimes \ell},\ell d) )\not= 0$.
Hence
$$\Big\{d\geq 0, \quad l\geq 0\Big\}
\subset Nef(\xc)
\subset\Big\{d\geq 0, \quad l\geq  -\frac{\deg\lambda_0}{(n+1)d_0}d\Big\}\subset Eff(\xc).
$$

\subsection{The cones in the generic case}
The ideas described here are due to Claire Voisin.
The key result is the following
\begin{lem}\label{lemme}
 Let $\yc\subset T\times P\to T$ be a family of complex algebraic ample hypersurfaces of dimension at least 3 of a projective manifold $P$.
Assume that $Y_0$ is irreducible, and that the nef cone and the pseudo-effective cone of its normalisation coincide. Then the nef cone and the pseudo-effective cone of a very general member of $\yc\to T$ do also coincide.
\end{lem}
\begin{proof}
The Picard group of any general member is induced by that of $P$ by Lefschetz theorem.
Take a numerical class $c$ in $NS(P)$.
Take a line bundle $\mathcal{L}$ on $P$ in the class $c$ whose restriction to a very general member $Y_t$ is effective.
There is a maximal Zariski closed subset $Z_c$ of $T$ such that the line bundle $\mathcal{L}_{\mid Y_t}$ is algebraically equivalent to an effective line bundle for all $t$ in $Z_c$. Define $Z$ to be the countable union of all those $Z_c$ that are strict in $T$.
Removing the countable union of images $Z'$ in $T$ of components that do not dominate $T$ of the Hilbert scheme of vertical curves in $\yc$, we can ensure that every curve $C$ in $Y_t$ for $t\in T-Z'$ deforms locally around $t$, and by properness of components of the Hilbert schemes, specialises to a curve $C_0$ at~$0$.

Take a $\tau\in T-Z-Z'$. Take a line bundle $\mathcal{L}\in\Pic(P)$ whose restriction to $Y_\tau$ is effective and a curve $C$ in $Y_\tau$. We have to check that $\deg\mathcal{L}_{\mid C}\geq 0$. 
The line bundle $\mathcal{L}_{\mid Y_t}$ is algebraically equivalent to an effective line bundle on the whole of $T$ and therefore $\mathcal{L}_{\mid Y_0}$ pulls back to a nef line bundle on the normalisation of $Y_0$, by hypothesis. 
Here we use the irreducibility of $Y_0$ to make sure that the gotten section do not identically vanish on some irreducible component of $Y_0$.
 For $\deg\nu^\star\mathcal{L}_{\mid\nu^{-1} C_0}\geq 0$, we infer using intersection theory for line bundles on the singular fibre $Y_0$ and especially the projection formula, that the integer $\deg \mathcal{L}_{\mid C}\geq 0$.
\end{proof}

In our setting this leads to the
\begin{prop}
Take a line bundle $\lambda_0\boxtimes\oc_{\pb^{n+1}}(d_0)$ on $B\times\pb^{n+1}$.
Assume that there is a rational function $f$ on $B$ such that $\deg\lambda_0$ is  $(n+1)\deg f$.
 If $\xc$ is very general in the linear system $|\lambda_0\boxtimes\oc_{\pb^{n+1}}(d_0)|$,
then $$Nef(\xc) = Eff(\xc) =\Big\{(l,d)/\quad d\geq 0, 
\quad l\geq  -\frac{\deg\lambda_0}{(n+1)d_0}d\Big\}.
$$
\end{prop}
\begin{proof}
Take a rational function $f~: B\to\pb^1$ of degree $m$
and a generic hypersurface $X$ of $\pb^{n+1}$ defined by the polynomial $F$ of degree $\mu$.
Construct the finite map gotten from Segre embedding and a general projection
$$\phi~:~B\times X\to\pb^1\times\pb^{n+1}\to\pb^{2n+1}\to\pb^{n+1}.$$
and the map $\Phi=(Id_B,\phi)~:~B\times X\to B\times\pb^{n+1}$.
Denote its image by $\xc_0$. The map $\pb^1\times\pb^{n+1}\to\pb^{2n+1}\to\pb^{n+2}$ is explicitly given in terms of coordinates by
\begin{multline*}
 [X_0:X_1]\times [Y_0:Y_1:\cdots :Y_{n+1}]
\mapsto [2X_1Y_0:X_0Y_0-X_1Y_1:X_0Y_1-X_1Y_2:\cdots\\\cdots : X_0Y_n-X_1Y_{n+1}:2X_0Y_{n+1}].
\end{multline*}
If $F(1,0,0,0)\not=0$ and $F$ is general, we can project to get a finite map 
\begin{align*}
 \pb^1\times X\to& \pb^{n+1}\\
 ([X_0:X_1],[Y_0:Y_1:\cdots :Y_{n+1}])\mapsto&
[X_0Y_0-X_1Y_1:X_0Y_1-X_1Y_2:\cdots\\& \cdots: X_0Y_n-X_1Y_{n+1}:2X_0Y_{n+1}].
\end{align*}
The equation of the image $\xc_0$ of $\Phi$ 
\begin{multline*}
 F(X_0^{n+1}U_{0}+X_0^{n}X_1U_1+X_0^{n-1}X_1^2U_2+\cdots +X^{n+1}_1\frac{U_{n+1}}{2}
:\cdots\\
\cdots 
: X_0^{n+1}U_{n-1}+X^{n}_0X_1U_n+X_0^{n-1}X_1^2\frac{U_{n+1}}{2}
:X^{n+1}_0U_n+X_0^{n}X_1\frac{U_{n+1}}{2}
:X_0^{n+1}\frac{U_{n+1}}{2})=0
\end{multline*}
is of bidegree $((n+1)m\mu,\mu)$.
Note that $n+1$ is the degree of the image of $\pb^1\times\pb^n$ (considered as a divisor in $\pb^1\times\pb^{n+1}$) by the Segre map to $\pb^{2n+1}$.

The nef cone of $B\times X$ is equal to its effective cone.
By lemma~\ref{lemme}, we infer that the same holds true for very general deformations of the image $\xc_0$, whose normalisation is $B\times X$.

We can now apply this to get more bidegrees than just those of type $((n+1)m\mu,\mu)$.
Take $\xc$ to be very general hypersurface of bidegree $((n+1)m,1)$ whose nef cone and the pseudo-effective cone coincide. Consider the Frobenius like finite morphism $\psi~:~\pb^{n+1}\to \pb^{n+1}$ gotten by raising homogeneous coordinates to the power $e$, and the hypersurface 
$\xc':=(Id_{\mid B},\psi)^{-1}(\xc)$.
It is a smooth ample hypersurface of $B\times \pb^{n+1}$ of bidegree $((n+1)m,e)$.
By Lefschetz theorem, its $\qb$-Neron Severi group coincide with that of $B\times \pb^{n+1}$.
Take a curve $C'$ and an effective divisor $D'$ in $\xc'$. Its multiple $eD'$ pulls back from an effective divisor $D$ in $\xc$, which is nef by hypothesis.
$$eD'\cdot C'=\psi^{-1}(D)\cdot C'=e^{n+1}D.\psi (C')\geq 0.$$
The hypersurface $\xc'$ may be not very general, but
applying the lemma again, we infer that the nef cone and the pseudo-effective cone of a very general hypersurface of bidegree $((n+1)m,e)$ coincide.
\end{proof}

\subsection{A nef line bundle on $\xc_1$}
By Leray-Hirsch theorem, the Picard group of $\xc_1=\pb(\cf_0)$ is the group 
$$\Pic\xc_1=\Pic\xc\oplus \zb\oc_{\xc_1}(1)=\Pic B\oplus \Pic \pb^{n+1}\oplus \zb\oc_{\xc_1}(1).$$
 Accordingly, we will use the notation $\oc_{\xc_1}(\lambda, d ; m_1)$.
The bundle $\Omega_{\pb^{n+1}}=\Lambda^{n} T_{\pb^{n+1}}\otimes K_{\pb^{n+1}}$
is a quotient of $\left(\Lambda^{n} \oc_{\pb^{n+1}}(1)^{\oplus n+1}\right)
\otimes K_{\pb^{n+1}}=\left(\Lambda^{n} \oc_{\pb^{n+1}}^{\oplus n+1}\right)
\otimes\oc_{\pb^{n+1}} (-2)$.
Hence, for $\Omega_B$ is globally generated, the quotient~(see \ref{inj})  $\cf_0\otimes\oc_{\pb^{n+1}}(2)$ and therefore the bundle $\oc_{\xc_1}(0,2;1)$ also are.

\subsection{A nef line bundle on $\xc_{k+1}$}
Generally, the bundle $\Omega_{\xc_k/\xc_{k-1}}=\Lambda^{n-1}T_{\xc_k/\xc_{k-1}}
\otimes K_{\xc_k/\xc_{k-1}}$ is a quotient of 
$$\Lambda^{n-1}\left(\pi_{k-1,k}^\star\cf_{k-1}^\star\otimes\oc_{\xc_k}(1)\right)
\otimes\oc_{\xc_k}(-n-1)\otimes\pi_{k-1,k}^\star\det \cf_{k-1}
=\pi_{k-1,k}^\star\cf_{k-1}\otimes\oc_{\xc_k}(-2).$$
Assuming that $\oc_{\xc_{k-1}}(\underline{m_{k-1}})$ and $\oc_{\xc_{k}}(\underline{m_{k-1}},1)$ are nef, we infer from 
the defining sequence of~$\cf_k$
\begin{eqnarray*}
0\to \oc_{\xc_{k}}(3\underline{m_{k-1}},3) &\to& \cf_k\otimes\oc_{\xc_k}(2)
\otimes\pi^\star_{k-1,k}\oc_{\xc_{k-1}}(3\underline{m_{k-1}})
 \\
&&\to\Omega _{\xc_k/\xc_{k-1}}\otimes\oc_{\xc_k}(2)\otimes\pi^\star_{k-1,k}
\oc_{\xc_{k-1}}(3\underline{m_{k-1}})\to 0
\end{eqnarray*}
setting 
 $\underline{m_k}:=(3\underline{m_{k-1}},2)=2(\underline{m_{k-1}},1)+(\underline{m_{k-1}},0)$
 that $\oc_{\xc_{k}}(\underline{m_{k}})$ and $\oc_{\xc_{k+1}}(\underline{m_{k}},1)$ are nef.
 We find that 
 $$L_k:=\oc_{\xc_{k}}(0,2\cdot 3^{k-1};2\cdot 3^{k-2},\cdots, 2\cdot 3^{2},2\cdot 3,2,1)$$ is nef of total degree $3^k$.


\section{Construction of differential equations}

\subsection{Definitions of Segre classes}
Recall that the total Segre class $s(E)$ of a complex vector bundle $E\to X$ of rank $r$ is defined in the following way : its component $s_i(E)$ of degree $2i$ is computed as $p_\star c_1(\oc_E(1))^{r-1+i}$, where $p~:~\pb(E)\to X$ is the variety of rank one quotients of $E$.
From this construction, one deduces that for a line bundle $L\to X$,
$$s_i(E\otimes L)=\sum_{j=0}^i \binom{r-1+i}{i-j} s_{j}(E)c_1(L)^{i-j}.$$
From Grothendieck defining relation for Chern classes
$$ c_r( p^\star E^\star\otimes\oc_E(1))=\sum_{i=0}^r p^\star c_i(E^\star)c_1(\oc_E(1))^{r-i}=0$$
one infers that total Segre class $s(E)$ is the formal inverse $c(E^\star)^{-1}$ of the total Chern class of the dual bundle $E^\star$. It is therefore multiplicative in short exact sequences.

\subsection{Computations on $\xc$}
Set on $B\times\pb^{n+1}$, $A:=pr_2^\star c_1(\oc_{\pb^{n+1}}(1))$, $B:=pr_1^\star c_1(\oc_B(1))$.
In particular, we can write $c_1(L_0)=dA+rB$.
We have the relations $A^{n+2}=0$, $B^2=0$, $A^{n+1}B=1$. 
Set on $\xc$, $\alpha:=R ^\star c_1(\oc_{\pb^{n+1}}(1))=\iota^\star A$, $\beta:=\pi^\star c_1(\oc_B(1))=\iota^\star B$.
We have the relations 
$$\alpha^{n+1}=c_1(L_0)A^{n+1}=r, \ \ \ \ \ \ \ \alpha^n\beta=c_1(L_0)A^nB=d.$$
Hence, $r$ is the degree of the map $R :=pr_2\circ\iota$, and $d$ is the degree of $X_b\subset \pb^{n+1}$.

From the relation~(\ref{inj}) and the Euler sequence on $\pb^{n+1}$, we infer that the total Segre class of $\cf_0=\Omega_\xc$ is 
\begin{eqnarray*}
s(\cf_0)&=&\pi^\star s(\Omega_B)R ^\star s(\Omega_{\pb^3})
\iota^\star s(L_0^\star)^{-1}\\
&=&\pi^\star s(\Omega_B) R ^\star s(\cb^{n+2}\otimes\oc_{\pb^{n+1}}(-1))
\iota^\star c(L_0)=\pi^\star s(\Omega_B) R ^\star c(\oc_{\pb^{n+1}}(1))^{-(n+2)}
\iota^\star c(L_0)\\
     &=& (1+\chi \beta) (1+\alpha)^{-(n+2)}  (1+d\alpha+r\beta).
\end{eqnarray*}
We find that the Segre classes of $\cf_0$ are polynomials in $(\alpha, \beta)$ with coefficients that are linear in $(r,d)$. In particular, 
\begin{eqnarray*}
s_1(\cf_0)&=& (d-n-2)\alpha+(r+\chi)\beta.
\end{eqnarray*}

\subsection{A recursion formula}
Recall the defining relation for the bundles $\cf_k$ on $\xc_k$
\begin{equation}\label{def}
0\to \oc_{\xc_{k}}(1) \to \cf_k 
 \to\Omega _{\xc_k/\xc_{k-1}}\to 0
\end{equation}
still valid for $k=0$, if we set $\xc_{-1}=B$, $\xc_0=\xc$, $\cf_0=\Omega_\xc$, 
$\oc_\xc(1)=\pi^\star \Omega_B$,
that is
$$0\to \pi^\star \Omega_B\to \Omega_\xc\to \Omega _{\xc/B}\to 0.$$
We will also need the relative Euler sequence on $\xc_k$
\begin{equation}\label{euler}
0\to\Omega_{\xc_k/\xc_{k-1}}
\to\pi_{k-1,k}^\star\cf_{k-1}\otimes\oc_{\xc_k}(-1)\to\oc_{\xc_k}\to 0.
\end{equation}

From the two previous sequences, we can compute the total Segre class of $\cf_k$ in terms of the Segre class of $\cf_{k-1}$ ($k\geq 1$).
Set $$\alpha_k:=c_1(\oc_{\cf_{k-1}}(1))=c_1(\oc_{\xc_{k}}(1)).$$

Remark to begin with, that the first Segre classes are easy to compute.
We find
\begin{equation}\label{s1}
s_1(\cf_k)=(\pi_{0,k})^\star s_1(\cf_0)
-n\left(\alpha_k+(\pi_{k-1,k})^\star\alpha_{k-1} +\cdots+(\pi_{1,k})^\star\alpha_{1}\right).
\end{equation}
In general,
\begin{eqnarray*}
s(\cf_k)
&=&s(\oc_{\xc_{k}}(1)) s(\Omega_{\xc_k/\xc_{k-1}})
=s(\oc_{\xc_{k}}(1))s(\pi_{k-1,k}^\star\cf_{k-1}\otimes\oc_{\xc_k}(-1))\\
&=&\sum_{\ell=0}^{(k+1)n+1} \sum_{i=0}^{\ell} s_{\ell-i}(\oc_{\xc_{k}}(1)) s_{i}(\Omega_{\xc_k/\xc_{k-1}})\\
&=&\sum_{\ell=0}^{(k+1)n+1} \sum_{i=0}^{\ell} s_{\ell-i}(\oc_{\xc_{k}}(1)) s_{i}(\pi_{k-1,k}^\star\cf_{k-1}\otimes\oc_{\xc_k}(-1))\\
&=&\sum_{\ell=0}^{(k+1)n+1} \sum_{i=0}^{\ell} \alpha_k^{\ell-i} \sum_{j=0}^i \binom{n+i}{i-j}\pi_{k-1,k}^\star s_{j}(\cf_{k-1})(-\alpha_k)^{i-j}\\
&=&\sum_{\ell=0}^{(k+1)n+1} \sum_{j=0}^\ell \pi_{k-1,k}^\star s_{j}(\cf_{k-1})\alpha_k^{\ell-j}
\sum_{i=j}^{\ell}(-1)^{i-j} \binom{n+i}{i-j}\\
&=&\sum_{\ell=0}^{(k+1)n+1} \sum_{j=0}^\ell \left[\sum_{i=0}^{\ell-j}(-1)^{i} \binom{n+j+i}{i}\right]
\pi_{k-1,k}^\star s_{j}(\cf_{k-1})\alpha_k^{\ell-j}.
\end{eqnarray*}
Defining the numbers 
$\lc_{e}^{f+e}:= \sum_{i=0}^{f}(-1)^{i} \binom{e+i}{e}$, we get 
\begin{equation*}\label{segre}
s_\ell(\cf_k)=\sum_{a+b=\ell}\lc_{n+a}^{n+\ell}\pi_{k-1,k}^\star s_{a}(\cf_{k-1})\alpha_k^{b}.
\end{equation*}

\subsection{Estimates for intersection numbers}
The idea comes from the reading of \cite{diverio}. 
Recall that the line bundle
$$L_k:=\oc_{\xc_{k}}(0,2\cdot 3^{k-1};2\cdot 3^{k-2},\cdots, 2\cdot 3^{2},2\cdot 3,2,1)$$ is nef on $\xc_k$. Its first Chern class is
$$l_k:=\alpha_k+2\pi_{k-1,k}^\star\alpha_{k-1}+6\pi_{k-2,k}^\star\alpha_{k-2}+\cdots+ 2\cdot 3^{j-1}\pi_{k-j,k}^\star\alpha_{k-j}+\cdots +2\cdot 3^{k-2}\pi_{1,k}^\star\alpha_1+2\cdot 3^{k-1}\pi_{0,k}^\star\alpha.$$

We are in position to prove
\begin{lem}\label{estimates}
 \begin{enumerate} For $r\gg d\gg 1$,
\item  
 $$s_1(\cf_0)^{n+1}\sim (n+2)rd^{n+1}.$$
 \item $$(\pi_{k-1,k})_\star l_k^{n+1}\geq \pi_{0,k-1}^\star s_1(\cf_0)$$
\item $$l_1^{m_1}l_2^{m_2}\cdots l_s^{m_s}\cdot\alpha\leq C(d^{n+1}+rd^n).$$
\end{enumerate}
\end{lem}
The output is that the leading numerical term comes from the relative canonical degree.
\begin{proof}
 \begin{enumerate}
\item Just compute 
\begin{eqnarray*}
s_1(\cf_0)^{n+1}&=& \left((d-n-2)\alpha+(r+\chi)\beta\right)^{n+1}\\
&\sim& d^{n+1}\alpha^{n+1}+(n+1)rd^n\alpha^n\beta=(n+2)rd^{n+1}.
\end{eqnarray*}
  \item Recall from the relation~(\ref{s1}) that
\begin{eqnarray*}
(\pi_{k-1,k})_\star l_k^{n+1}&=&s_1(\cf_{k-1})\\
&&+(n+1)(2\alpha_{k-1}+6\alpha_{k-2}+\cdots+ 2\cdot 3^{j-1}\alpha_{k-j}+\cdots +2\cdot 3^{k-2}\alpha_1+2\cdot 3^{k-1}\alpha)\\
&=&\pi_{0,k-1}^\star s_1(\cf_0) -n(\alpha_{k-1}+\alpha_{k-2}+\cdots+\alpha_{1})\\
&&+(n+1)(2\alpha_{k-1}+6\alpha_{k-2}+\cdots+ 2\cdot 3^{j-1}\alpha_{k-j}+\cdots +2\cdot 3^{k-2}\alpha_1+2\cdot 3^{k-1}\alpha)\\
&=&\pi_{0,k-1}^\star s_1(\cf_0)+(n+2)\alpha_{k-1}+(5n+6)\alpha_{k-2}+\cdots\\&&+ \left((n+1)2\cdot 3^{j-1}-n \right)\alpha_{k-j}
 +\left((n+1)2\cdot 3^{j}-n \right)\alpha_{k-j-1}+\cdots
\\&& +\left((n+1)2\cdot 3^{k-2}-n\right)\alpha_1+2\cdot3^{k-1}\alpha.
\end{eqnarray*}
The claim follows from the inequalities 
$(n+1)2\cdot 3^{j}-n\geq 3\left[(n+1)2\cdot 3^{j-1}-n\right]$ 
that ensure the nefness of $(\pi_{k-1,k})_\star l_k^{n+1}-\pi_{0,k-1}^\star s_1(\cf_0)$.
\item It follows from the recursion formula that, computed in $\xc$ of dimension $n+1$,
\begin{eqnarray*}
l_1^{m_1}l_2^{m_2}\cdots l_s^{m_s}\cdot\alpha&=&\sum_{k\leq n}C_Is_{i_1}(\cf_0)s_{i_2}(\cf_0)
\cdots s_{i_k}(\cf_0)\cdot\alpha. 
\end{eqnarray*}
Recall that the Segre classes of $\cf_0$ are polynomials in $(\alpha, \beta)$ whose  coefficients are linear in $(r,d)$.
\begin{eqnarray*}
l_1^{m_1}l_2^{m_2}\cdots l_s^{m_s}\cdot\alpha&=& P(r,d)\alpha^{n+1}+Q(r,d)\alpha^n\beta=P(r,d) r+Q(r,d) d
\end{eqnarray*}
where $P$ and $Q$ are polynomials in $(r,d)$ of degree less or equal to $n$.
\end{enumerate}
\end{proof}

\subsection{Final argument}
We choose $\kappa=n+1$. We work on $\xc_\kappa$ with the fractional line bundle
$$A=L_{n+1}\otimes L_n\otimes\cdots\otimes L_j\otimes\cdots\otimes L_1\otimes \left[\oc_{\pb^{n+1}}(1)\otimes \oc_B(-\frac{r}{(n+1)d})\right]$$
and we choose $B$ so that $L:=A\otimes B^{-1}$ has negative component along $Pic(\xc/B)$,
that is, for some fixed positive rational number $x$, 
$$B=\oc_{\pb^{n+1}}\left(2\cdot 3^{n+1-1}+\cdots +2\cdot 3^{j-1}+\cdots +2+1+x\right)=\oc_{\pb^{n+1}}(3^{n+1}+x).$$
For $\kappa=n+1$, we have $\dim \xc_\kappa=\kappa (n+1)$. 
Hence, for we only omit intersections of nef classes,
\begin{eqnarray*}
{A^{\dim\xc_\kappa}} 
&=& \left(l_\kappa+l_{\kappa-1}+\cdots +l_{1}+\left(\alpha-\frac{r}{(n+1)d}\beta\right)\right)^{\dim \xc_\kappa}\\
&\geq &l_\kappa^{(n+1)}l_{\kappa-1}^{(n+1)}\cdots l_{1}^{(n+1)}\\
&\geq& \pi_{0,\kappa-1}^\star s_1(\cf_0)\cdot l_{\kappa-1}^{(n+1)}\cdots l_{1}^{(n+1)}\\
&\geq& \pi_{0,\kappa-2}^\star s_1(\cf_0)^2\cdot l_{\kappa -2}^{(n+1)}\cdots l_{1}^{(n+1)}\\
&&\vdots\\
&\geq & s_1(\cf_0)^{\kappa}\sim (n+2)rd^{n+1}
\end{eqnarray*}
thanks to lemma~\ref{estimates}.
On the other hand, thanks to the same lemma,
\begin{eqnarray*}
{A^{\dim\xc_k-1}\cdot B} &=&\sum C_M l_1^{m_1}l_2^{m_2}\cdots l_s^{m_s}\cdot\alpha\leq C(d^{n+1}+rd^n).
\end{eqnarray*}
Fix $\rho~:~B_\rho\to B$.
If $r$ and $d$ are large enough so that $A^{\dim\xc_\kappa}-A^{\dim\xc_k-1}\cdot B>0$
and so that the inequality in Schwarz lemma
$\frac{\chi_\rho}{\deg\rho}\sum_{j=1}^{n+1}(3^{j}-2\cdot 3^{j-1})=\frac{\chi_\rho}{\deg\rho}\frac{3^{n+1}-1}{2} <\frac{r}{(n+1)d}$
is fulfilled, then the line bundle $A\otimes B^{-1}$ is big and the sections of its powers provide equations for the jets of sections of the family $\xc_\rho\to B_\rho$.
This ends the proof of theorem~\ref{theo1}. We will in fact need a more precise version.
\begin{Def}
 A family $\pi~:\xc\to B$ of hypersurfaces in $\pb^{n+1}$ is said to be ``moving enough'' if it is given inside $B\times\pb^{n+1}$ by a section of a line bundle $\lambda\boxtimes\oc_{\pb^{n+1}}(d)$ where $\lambda$ is a line bundle on $B$ whose degree, called the variation of $\pi$, is large and equal to $(n+1)$-times the degree of a rational function on $B$.
\end{Def}

\begin{thm}\label{theo2}
 For all fixed positive $\delta$, for $r$ and $d$ large enough, there exists an $m_0$ such that all sections of general ``moving enough'' families of hypersurfaces in $\pb^{n+1}$ of degree $d$ and variation $r$ fulfil a differential equation of order $n+1$, given by a section of a line bundle 
$$\oc_{\xc_{n+1}}(\underline{m})\otimes \oc_{\pb^{n+1}}(-\mid \underline{m}\mid\delta )\otimes 
\oc_B(-\mid \underline{m}\mid\frac{r}{(n+1)d})$$
where $\mid\underline{m}\mid=m_0$.
\end{thm}

\subsection{Height inequalities} 
We now look for a statement that incorporates the dependence in the ramified cover $B_\rho\to B$.
We work on $\xc_{n+1}$ with 
$$A=L_{n+1}\otimes L_n\otimes\cdots\otimes L_j\otimes\cdots\otimes L_1$$
and we choose $B$ so that $L:=A\otimes B^{-1}$ has negative component on $Pic(\xc/B)$,
that is 
$$B=\oc_{\pb^{n+1}}(3^{n+1}-1+x).$$
The previous computations show that $A-B$ is big for large enough $r$ and $d$.
As a result, we obtain
\begin{thm}
 Fix a positive integer $x$. For large enough $r$ and $d$, for every family $\xc$ gotten by section of $\oc_{\pb^{n+1}}(d)\boxtimes\oc_B(r)$ , there exists a proper algebraic set $\mathcal{Y}\subset\xc_{n+1}$ such that for every finite ramified cover $\rho~:~B_\rho\to B$ and every section $s$ of $\rho^\star\xc\to B_\rho$ whose $(n+1)$-th order jet do not lie in $\mathcal{Y}$, the following height inequality holds
$$h(s(B))=\frac{s(B_\rho)\cdot\oc_{\pb^{n+1}}(1)}{\deg\rho}\leq \frac{3^{n+1}-1}{2x} \frac{\chi_\rho}{\deg\rho}.$$
\end{thm}
This is an analog of the first part of Vojta's work~\cite{vojta}. The deepest part dealing with sections having $(n+1)$-jet inside $\mathcal{Y}$ would require an analog of Jouanoulou's result on foliations, that seems out of reach now.

\section{Non-Zariski density}
We follow the ideas of Siu~\cite{siu}, described in details in ~\cite{mdr}.
Let $B$ be a compact complex curve, $\oc_B(r)$ a holomorphic line bundle on $B$.
Consider the linear system $\mid\oc_B(r)\boxtimes\oc_{\pb^{n+1}}(d)\mid=\pb^N$ on $B\times{\pb^{n+1}}$ whose each element represents a family $\xc\to B$ of degree $d$ hypersurfaces in $\pb^{n+1}$ parametrised by $B$ with variation $\oc_B(r)$.
Consider the associated universal family
$$\begin{array}{ccc}
   \mathfrak{X}&\subset&\pb^N\times B\times\pb^{n+1}\\
\Pi\downarrow&\swarrow&\\
\pb^N&&
  \end{array}$$
 We will denote by $\mathfrak{X}_\kappa$ the $\kappa$-jets space of sections of the families $\xc\to B$.

\subsection{Proof using vector fields on universal families}

Consider a family $\Pi^{-1}(A)=(\pi~:~\xc^A=\xc\to B)$ of degree $d$ hypersurfaces in $\pb^{n+1}$ parametrised by $B$ with variation $\oc_B(r)$. Consider a section $s~:~B\to\xc$ of $\pi$ and a non-zero section $\sigma$ of the line bundle 
$\mathcal{L}_{-\lambda,-\delta,\underline{m}}
:=\oc_{\xc_\kappa}(\underline{m})\otimes \oc_{\pb^{n+1}}(-\mid\underline{m}\mid\delta )
\otimes\lambda^{-\mid \underline{m}\mid} $ on $\xc_\kappa$ where, for $r$ and $d$ are assumed to be large, we can impose $\delta>0$ and the inequality of Schwarz lemma
$\displaystyle\deg \lambda>\frac{\chi_\rho}{\deg\rho}$.
The pushforward $(\pi_{\kappa,0})_\star \sigma$ is a non-zero section of the vector bundle
$(\pi_{\kappa,0})_\star \mathcal{L}_{-\lambda,-\delta,\underline{m}} =E_{-\lambda,-\delta,\underline{m}}\to\xc$.
Constant sections are those whose first order jet lies inside $\{z'_1=z'_2=\cdots=z'_n=0\}$.
We can now prove the precise version of the main theorem.
\begin{thm}
 If $r$ and $d$ are large enough, $\pi$ (i.e. $A$) is generic, $\mid\underline m\mid$ meets  the requirements of holomorphic Morse inequalities and $s$ is not constant, then 
$$s(B)\subset  \textrm{Zero}\left((\pi_{\kappa,0})_\star \sigma\right)\subset\xc.$$
\end{thm}
\begin{proof} We only sketch the proof, the details being close to that given in~\cite{mdr}.
We argue by contradiction and assume that there exists a $b_0$ in $B$ where 
$(\pi_{\kappa,0})_\star \sigma (s(b_0))\not=0$.
Take another view point on the section $\sigma$ and view it as a meromorphic function 
$$\begin{array}{cccc}
  \xc_\kappa&\to&\cb\\
\zeta_\kappa&\mapsto&\sum_{wl(I)=m}q_I(b,z)(z'(\zeta_\kappa))^{i_1}\cdots
(z^{(\kappa)}(\zeta_\kappa))^{i_\kappa}
 \end{array}
$$ 
where the $q_I(b,z)$ are meromorphic functions on $\xc$, holomorphic when viewed as sections of $\oc_B(\lambda)\boxtimes\oc_{\pb^{n+1}}(-\delta)$.
The assumption $(\pi_{\kappa,0})_\star \sigma (s(b_0))\not=0$ translates into the existence of a multiindex $I_0$ of weighted length $m$ such that $q_{I_0}(s_k(b_0))\not=0$.

 By the genericity assumption on $\Pi^{-1}(A)=\pi$ we may extend the section 
$(\pi_{\kappa,0})_\star\sigma$ to a
section of $\mathfrak{E}_{-\lambda,-\delta,\underline{m}}\to\mathfrak{X}$ on a neighbourhood of  $\xc$  in $\mathfrak{X}$.

\begin{prop}\label{vector fields}
 Every vector of
$$T\left(\xc^A_\kappa/\xc^A\right)_{(s_\kappa(b_0))}
\subset\left(T\xc^A_\kappa\right)_{(s_\kappa(b_0))}
=T\left(\mathfrak{X}_\kappa/\pb^N\right)_{(A,s_\kappa(b_0))}\subset \left(T\mathfrak{X}_\kappa\right)_{(A,s_\kappa(b_0))}$$
outside the set  $\Pi_{\kappa,1}^{-1}\{z'_1=z'_2=\cdots=z'_n=0\}$
is the value of a meromorphic vector field on 
$\Pi_{\kappa,0}^{-1}(U_A)\subset \mathfrak{X}_\kappa$ 
holomorphic when viewed with values in $\Pi_{\kappa,0}^\star\oc_{\pb^{n+1}}(\kappa^2+2\kappa)$. 
\end{prop}

Take it for granted until the next subsection.
When we differentiate the meromorphic function $\sigma$ with the gotten meromorphic vector fields at most $\mid\underline m\mid$-times, we get meromorphic functions on a neighbourhood of $\xc$  in $\mathfrak{X}$ that in turn can be viewed as a section of 
$\mathcal{L}_{-\lambda,-\delta+(\kappa^2+2\kappa),\underline{m}}$.
If $-\delta+(\kappa^2+2\kappa)$ is still negative 
then $s_\kappa$ has to fulfil this new equation. Having chosen the vector fields in a suitable way, thanks to the proposition~\ref{vector fields}, this contradicts $q_{I_0}(s_k(b_0))\not=0$.
\end{proof}

The proof of the main theorem is now ended by the following.
Constant sections have null height because they are also constant in the product $B\times\pb^{n+1}$. If their images would dominate the total space, the arguments of Maehara and Moriwaki~\cite{moriwaki2} using positivity of direct images of pluricanonical line bundles would show that the family has to be birationally trivial. 
 
\subsection{Constructing vector fields on universal families}

In homogeneous coordinates, having chosen a basis for $\cb^{n+2}$, 
the corresponding basis of monomials for $\mid\oc_{\pb^{n+1}}(d)\mid$, and a basis $(\Phi_\beta)_\beta$ for $\mid \oc_B(r)\mid$, the hypersurface $\mathfrak{X}$ of $\pb^N\times B\times\pb^{n+1}$ is defined by the equation
$$\sum_{\alpha,\beta} \mathfrak{A}_\alpha^\beta \Phi_\beta \mathfrak{Z}^\alpha=0.$$
On the open set $\{\mathfrak{A}_{0,d,0,0,\cdots,0}^0\not=0\}\times\{\Phi_0(b)\not=0\}\times\{\mathfrak{Z}_0\not=0\}$
the equation rewrites in inhomogeneous coordinates
$$\mathcal{F}=z_1^d+\sum_{{\alpha\in\nb^{n+1}, \mid\alpha\mid\leq d\atop \alpha\not=(d,0,0,\cdots,0)}\atop \beta\geq 1}a_\alpha^\beta\varphi_\beta(b) z^\alpha=0.$$
Over this open set, the natural open set of the $\kappa$-jets space $\mathfrak{X}_\kappa$ of sections of the families $\xc\to B$ is given inside 
$\cb^N\times U\times\cb^{n+1}\times
\underbrace{\cb^{n+1}\times\cdots\times\cb^{n+1}}_{\kappa\textrm{ times }}$
in terms of the operator 
$$\mathfrak D:=\frac{\partial}{\partial t}+ \sum_{\lambda=0}^\kappa\sum_{j=1}^{n+1}z_j^{(\lambda+1)}\frac{\partial}{\partial z_j^{(\lambda)}}$$
 by the following set of equations
\begin{eqnarray*}
 \sum_{{\alpha\in\nb^{n+1}\atop \mid\alpha\mid\leq d}\atop \beta\geq 1}a_\alpha^\beta\varphi_\beta(t) z^\alpha=
\mathfrak D\left(\sum a_\alpha^\beta\varphi_\beta(t) z^\alpha\right) 
=\mathfrak D^2\left(\sum a_\alpha^\beta\varphi_\beta(t) z^\alpha\right) =&\cdots\\
\cdots=
\mathfrak D^\kappa\left(\sum a_\alpha^\beta\varphi_\beta(t) z^\alpha\right) &=0.
\end{eqnarray*}
Those are the equations one infers from the derivatives of the relation
$\sum a_\alpha^\beta\varphi_\beta(t) z^\alpha(t)$ fulfilled by sections $t\mapsto (t,z_1(t),\cdots, z_{n+1}(t))$ of a family $\Pi^{-1}({A})$,
after substituting $z_j^{(\lambda)}:=\frac{\partial^\lambda z_j(t)}{\partial t^\lambda}$.
Denote the partial sum $\sum_{\alpha\in\nb^{n+1},\mid\alpha\mid\leq d}a_\alpha^\beta z^\alpha$ by $\mathcal{F}_\beta$. The equations for a vector field $T$ of the special shape
$T:=\left(\sum_\beta T_\beta\right)+ T_z$, 
where $T_\beta:=\sum_\alpha A_\alpha^\beta\frac{\partial}{\partial a_\alpha^\beta}$
and $T_z:=\sum_{\lambda=0}^\kappa\sum_{j=1}^{n+1}P_j^\lambda\frac{\partial}{\partial z_j^{(\lambda)}}$,
to be tangent to~$\mathfrak{X}_\kappa$ 
rewrite, thanks to Leibniz formula and the fact that when $\beta\not=\gamma$, $T_\beta\cdot D^a\mathcal{F}_\gamma=0$, in terms of the operator 
$D :=\sum_{\lambda=0}^\kappa\sum_{j=1}^{n+1}z_j^{(\lambda+1)}\frac{\partial}{\partial z_j^{(\lambda)}}$ as
\begin{eqnarray*}
 \sum_{\beta\geq 1}\varphi_\beta(t)(T_\beta+ T_z)\cdot\mathcal{F}_\beta&=&0\\
\sum_{\beta\geq 1}\varphi_\beta(t)(T_\beta+ T_z)\cdot D\left(\mathcal{F}_\beta\right)
+ \varphi'_\beta(t)(T_\beta+ T_z)\cdot\mathcal{F}_\beta&=&0\\
\sum_{\beta\geq 1}\varphi_\beta(t)(T_\beta+ T_z)\cdot D^2\left(\mathcal{F}_\beta \right) 
+2\varphi'_\beta(t)(T_\beta+ T_z)\cdot D\left(\mathcal{F}_\beta\right)
+ \varphi''_\beta(t)(T_\beta+ T_z)\cdot\mathcal{F}_\beta&=&0\\
&\vdots&\\
\sum_{\beta\geq 1}\sum_{a=0}^\kappa\left( {\kappa\atop a} \right)\varphi^{(\kappa-a)}_\beta (T_\beta+ T_z)\cdot D^a\left(\mathcal{F}_\beta\right) &=&0.
\end{eqnarray*}
A set of sufficient conditions is therefore
$$\forall\beta, \ \ (T_\beta+ T_z)\cdot\mathcal{F}_\beta
=(T_\beta+ T_z)\cdot D\left(\mathcal{F}_\beta\right)
=\cdots=(T_\beta+ T_z)\cdot D^\kappa\left(\mathcal{F}_\beta \right)=0$$
reducing to the absolute case.
Note however that, for theorem~\ref{theo2} provides us with a differential equation of order $n+1$, we need to consider $(n+1)$-th order jets of hypersurfaces in $\pb^{n+1}$, whereas the by now well settled results are for $n$-th order jets.

\subsection{Constructing vector fields in the absolute case}
We follow the ideas of Siu, P\u aun, Rousseau and Merker. 
For notational simplicity, we will replace $\beta$ by a dot in the following.
The exponents in brackets will be relative to the absolute operator $D$.

Write $(T_\cdot+T_z)\cdot D^{l+1}(\mathcal{F}_\cdot)=[T_\cdot+T_z,D]D^l(\mathcal{F}_\cdot)
+D\left((T_\cdot+T_z)\cdot D^l(\mathcal{F}_\cdot)\right)$ to infer that
a set of sufficient conditions for the special vector field $T_\cdot+T_z$ to contribute to a tangent to $\mathfrak{X}_\kappa$ is 
\begin{eqnarray*}
 (T_\cdot+T_z)\cdot\mathcal{F}_\cdot=[T_\cdot+T_z,D]\cdot\mathcal{F}_\cdot
=[T_\cdot+T_z,D]\cdot D(\mathcal{F}_\cdot)=\cdots\\
\cdots=[T_\cdot+T_z,D]\cdot D^{\kappa-1}(\mathcal{F}_\cdot)=0.
\end{eqnarray*}
We will now further restrict the shape of the chosen vector field to simplify its commutator with $D$.
\begin{lem}
Let $A_\alpha^\cdot$ and $P$ be functions in the $(z_i^{(\lambda)})$ variables.
The commutator of the very special vector field $T_\cdot+T_z=\sum_\alpha A_\alpha^\cdot\frac{\partial}{\partial a_\alpha^\cdot}
+\sum_{\lambda=0}^\kappa P^{(\lambda)}\frac{\partial}{\partial z_j^{(\lambda)}}$
  with $D$ is 
$$[T_\cdot+T_z,D]=-\sum_\alpha (A_\alpha^\cdot)'\frac{\partial}{\partial a_\alpha^\cdot} -P^{(\kappa+1)}\frac{\partial}{\partial z_j^{(\kappa)}}.$$
\end{lem}
\enlargethispage{1cm}
\begin{proof}
 Simply check that
\begin{eqnarray*}
T_\cdot(D)=0&&
D(T_\cdot)=\sum_\alpha (A_\alpha^\cdot)'\frac{\partial}{\partial a_\alpha^\cdot}\\
 T_z(D)=\sum_{\lambda=1}^\kappa P^{(\lambda)}\frac{\partial}{\partial z_j^{(\lambda-1)}}
=\sum_{\lambda=0}^{\kappa-1} P^{(\lambda+1)}\frac{\partial}{\partial z_j^{(\lambda)}}&&
D(T_z)=\sum_{\lambda=0}^\kappa P^{(\lambda+1)}\frac{\partial}{\partial z_j^{(\lambda)}}.
\end{eqnarray*}
\end{proof}

We infer that a set of sufficient conditions for the very special vector field $T_\cdot+T_z$ to contribute to a tangent vector field to $\mathfrak{X}_\kappa$ is 
\begin{eqnarray*}
 \sum_\alpha A_\alpha^\cdot z^\alpha+P\sum_\alpha a_\alpha^\cdot\frac{\partial z^\alpha}{\partial z_j}=0\\
-\sum_\alpha (A_\alpha^\cdot)' z^\alpha=
-\sum_\alpha (A_\alpha^\cdot)' (z^\alpha)'=
-\sum_\alpha (A_\alpha^\cdot)'(z^\alpha)^{(2)}=\cdots\\
\cdots=-\sum_\alpha (A_\alpha^\cdot)'(z^\alpha)^{(\kappa-1)}=0
\end{eqnarray*}
or equivalently, using the formula
$D^{l+1}\left(\sum_\alpha A_\alpha z^\alpha\right)
=\sum_\alpha A_\alpha (z^\alpha)^{(l+1)} + \sum_\alpha A'_\alpha(z^\alpha)^{(l)} 
+ \sum_{k=0}^{l-1}D^{l-k}\left(\sum_\alpha A'_\alpha(z^\alpha)^{(k)}\right)$,
\begin{eqnarray}\label{1}
 \nonumber\sum_\alpha A_\alpha^\cdot z^\alpha
&=&-P\sum_\alpha a_\alpha^\cdot\frac{\partial z^\alpha}{\partial z_j}\\
\nonumber\sum_\alpha A_\alpha^\cdot (z^\alpha)'&=& \left(-P\sum_\alpha a_\alpha^\cdot\frac{\partial z^\alpha}{\partial z_j}\right)'\\
\nonumber\sum_\alpha A_\alpha^\cdot (z^\alpha)^{(2)}
&=&\left(-P\sum_\alpha a_\alpha^\cdot\frac{\partial z^\alpha}{\partial z_j}\right)^{(2)}\\
\sum_\alpha A_\alpha^\cdot(z^\alpha)^{(3)}
&=&\left(-P\sum_\alpha a_\alpha^\cdot\frac{\partial z^\alpha}{\partial z_j}\right)^{(3)}\\
\nonumber&\vdots&\\
\nonumber\sum_\alpha A_\alpha^\cdot(z^\alpha)^{(\kappa)}
&=&\left(-P\sum_\alpha a_\alpha^\cdot\frac{\partial z^\alpha}{\partial z_j}\right)^{(\kappa)}
\end{eqnarray}
or also
\begin{eqnarray}\label{2}
 \nonumber\sum_\alpha A_\alpha^\cdot z^\alpha=-P\sum_\alpha a_\alpha^\cdot\frac{\partial z^\alpha}{\partial z_j}\\
\sum_\alpha (A_\alpha^\cdot)' z^\alpha=
\sum_\alpha (A_\alpha^\cdot)'' z^\alpha=
\sum_\alpha (A_\alpha^\cdot)^{(3)}z^\alpha=\cdots\\
\nonumber\cdots=\sum_\alpha (A_\alpha^\cdot)^{(\kappa)}z^\alpha=0.
\end{eqnarray}

When $P$ is of degree less than $2$, the first equation in the set~(\ref{2}) can be fulfilled with constant $A^\cdot_\alpha$, 
making the other equations tautological.

When $P$ is of the form $P=z_i^k$, because the only non-zero term in the right hand side of ~(\ref{2}) can be written as 
$$\sum_{\mid\beta\mid\leq d} b^{\cdot}_\beta z^\beta
+\sum_{\ell=1}^{k-1}
\sum_{\atop \mid\beta\mid= d} b^{\cdot\ell}_\beta z^{\beta+\ell\epsilon_i},$$
we look for $A^\cdot_\alpha$ in the form 
$$A^\cdot_\alpha:=
\sum_{\gamma, \mid\gamma\mid\leq \kappa\atop \mid\alpha+\gamma\mid\leq d} A^{\cdot\gamma}_\alpha z^\gamma
+\sum_{\ell=1}^{\min (\alpha_i,k-1)}
\sum_{\gamma, \mid\gamma\mid\leq\kappa \atop \mid\alpha+\gamma-\ell\epsilon_i\mid=d} A^{\cdot\ell,\gamma}_\alpha z^{\gamma}.$$
Note that for $\alpha_i\geq\ell$, the multiindex $\alpha+\gamma-\ell\epsilon_i$ is non-negative.
Then the set~(\ref{2}) rewrites, after recursive simplifications of all terms involving a $z_l^{(k)}$-variable with $k>1$, as a set of systems, one for each multiindex $\mu+\ell\epsilon_i$ where $\mu$ is a multiindex of length $\mid\mu\mid\leq d$ when $\ell=0$, or $\mid\mu\mid= d$ when $1\leq\ell\leq k-1$. They have disjoint sets of indeterminates $(A^{\cdot\ell,\gamma}_{\alpha})
_{\alpha+\gamma=\mu+\ell\epsilon_i\atop
\mid\alpha\mid\leq d,\mid\gamma\mid\leq\kappa}$. 
Note that for $\alpha_i\geq\ell$ the equality $\alpha+\gamma=\mu+\ell\epsilon_i$ implies $\gamma\leq\mu$.
The coefficient on the row indexed by the multiindex 
$\delta\leq \mu$ of length $\mid\delta\mid\leq\kappa$ and the column indexed by
 $\gamma\leq\mu$ of length $\mid\gamma\mid\leq \kappa$ is $z^{\mu+\ell\epsilon_i-\gamma}\frac{\partial^{\mid\delta\mid}z^{\gamma}}{(\partial z)^\delta} $. 
Its determinant is checked, as in~\cite{paun}, to be non-zero, for otherwise there would exist a non-zero polynomial of multidegree less or equal to $\mu$ and total degree less or equal to $\kappa$ with all derivatives of order less or equal to $\mu$ and total order less or equal to $\kappa$ vanishing.
Let the polynomial $P$ run over the set of polynomials in $z_i$ of degree less or equal to $\kappa$. 
Over the set $\{z_i'\not=0\}$, the determinant, computed by induction using 
$(z_i^j)^{(l)}=(jz_i^{j-1}z_i')^{(l-1)}=j\sum_{a=0}^{l-1}\binom{l-1}{a} (z_i^{j-1})^{(a)}z_i^{(l-a)}$ and combinaison of rows,
$$\det\left(  \begin{array}{ccccc}
 1&z_i&z_i^2&\cdots&z_i^\kappa\\   1'&(z_i)'&(z_i^2)'&\cdots&(z_i^\kappa)'\\  &&\vdots&&\\&&\vdots&&\\
 (1)^{(\kappa)}&(z_i)^{(\kappa)}&(z_i^2)^{(\kappa)}&\cdots&(z_i^\kappa)^{(\kappa)}  \end{array} \right)=1!2!\cdots\kappa !(z'_i)^\frac{\kappa(\kappa+1)}{2}$$
 does not vanish.
This shows that  every vector of
$$T\left(\xc^A_\kappa/\xc^A\right)_{(s_\kappa(b_0))}
\subset\left(T\xc^A_\kappa\right)_{(s_\kappa(b_0))}
=T\left(\mathfrak{X}_\kappa/\pb^N\right)_{(A,s_\kappa(b_0))}\subset \left(T\mathfrak{X}_\kappa\right)_{(A,s_\kappa(b_0))}$$
is, up to ``horizontal vectors'', the value of a meromorphic vector field on 
$\Pi_{\kappa,0}^{-1}(U_A)\subset \mathfrak{X}_\kappa$ 
holomorphic when viewed with values in $\Pi_{\kappa,0}^\star\oc_{\pb^{n+1}}(\kappa)$.

For ``horizontal vectors'' (i.e. when $P=0$), we use the set~(\ref{1}).
By Cramer formulae, over the set $\{z_i'\not=0\}$, for any given set of $(A^\cdot_\alpha)_{\mid\alpha\mid\leq\kappa\atop\alpha\not=l\epsilon_i}$
there exists $(A^\cdot_{l\epsilon_i})_l$ that fulfil the previous equations.
Their pole order is less or equal to $\kappa^2+2\kappa$.
The missing directions $(A^\cdot_\alpha)_{\mid\alpha\mid>\kappa}$ are obtained with even smaller pole order,
by considering some universal relations in the differential algebra of polynomials.
Details for this last paragraph can be red in~\cite{merker}.

\section{Appendix : Using Morse inequalities for families of surfaces}
We check that in the case of surfaces, the bound $\kappa=n+1$ is optimal to find differential equations using holomorphic Morse inequalities. 

Remark first that the numbers $\lc$, that appeared in the recursion formula for Segre classes of the bundles $\cf_k$, can easily be computed writing Pascal triangle. 
They also fulfil the relation 
$$\lc_e^{f}-\lc_{e+1}^f=\lc_{e+1}^{f+1}.$$

$$\begin{array}{r|rrrrrrrrrr}
\lc_e^f&e=0&e=1&e=2&e=3&e=4&e=5&e=6&e=7&e=8&e=9\\
\hline\\
f=0&1&&&&&&&&&\\
f=1&0&1&&&&&&&&\\
f=2&1&-1&1&&&&&&&\\
f=3&0&2&-2&1&&&&&&\\
f=4&1&-2&4&-3&1&&&&&\\
f=5&0&3&-6&7&-4&1&&&&\\
f=6&1&-3&9&-13&11&-5&1&&&\\
f=7&0&4&-12&22&-24&16&-6&1&&\\
f=8&1&-4&16&-34&46&-40&22&-7&1&\\
f=9&0&5&-20&50&-80&86&-62&29&-8&1\\
\end{array}$$

\subsection{On $\xc_1$}
We choose $\ep$ to be equal to the bound we found when computing the generic nef cone of $\xc$, that is $\ep:=\frac{r}{3d}$. Then, we take $A=\oc_{\xc_1}(0,2;1)\otimes \oc_{\xc}(-\ep x, x)$ and $B=\oc_{\xc}(0, 2+x)$
with first Chern class $$a=\alpha_1+2\alpha + x(\alpha-\ep \beta)=\alpha_1+(2+x)\alpha -x\ep \beta$$ and 
$$b=(2+x)\alpha.$$ 
We find
\begin{eqnarray*}
 A^5-5A^4B&=&(\alpha_1-\ep x\beta)^5-10(\alpha_1-\ep x\beta)^3(2+x)^2\alpha^2-20(\alpha_1-\ep x\beta)^2(2+x)^3\alpha^3\\
&=& s_3-5\ep xs_2\beta-10(2+x)^2s_1\alpha^2-20(2+x)^3\alpha^3+30\ep (2+x)^2x\alpha^2\beta
\end{eqnarray*}
whose dominant term 
\begin{eqnarray*}
 \left[-4\chi+20\ep x\right]d^2+20\left[1-(2+x)^2\right]rd&=&-4\chi d^2-20\left[3+11/3x+x^2\right]rd
\end{eqnarray*}
is negative.

\subsection{On $\xc_2$}
Here 
we take $A=\oc_{\xc_2}(0,6 ;2,1)\otimes \oc_{\xc_1}(0, 2y;y)\otimes\oc_{\xc}(-\ep x, x)$ and $B=\oc_{\xc}(0, 6+2y+x)$
with first Chern class 
\begin{eqnarray*}
 a&=&(\alpha_2+2\alpha_1+6\alpha)+y(\alpha_1+2\alpha) + x(\alpha-\ep \beta)\\
&=&\alpha_2+(2+y)\alpha_1+(6+2y+x)\alpha-\ep x\beta
\end{eqnarray*}
and $b=(6+2y+x)\alpha$.
The bundle $A\otimes B^{-1}$ is $\oc_{\xc_2}(-\ep x, 0;2+y,1)$. 

We compute only the term $(A^7-7A^6B)_{dom}$ in $A^7-7A^6B$ of degree $3$ in $(r,d)$.
From the computation of the direct images on $\xc$ of the Segre classes of $\cf_1$, and from the Segre numbers of $\cf_0$ on $\xc$ we infer that the contributions have to contain a part in $s_1s_2$ or $s_1^2$ and  should therefore contain only one power of $\alpha$ or $\beta$.
We find that the dominant term is, viewed in $\xc_2$
\begin{eqnarray*}
 (A^7-7A^6B)_{dom} &= &[\alpha_2+(2+y)\alpha_1]^7+7[\alpha_2+(2+y)\alpha_1]^6[(6+2y+x)\alpha-\ep x\beta]\\
&&-7[\alpha_2+(2+y)\alpha_1]^6(6+2y+x)\alpha\\
&=&[\alpha_2+(2+y)\alpha_1]^7-7\ep x[\alpha_2+(2+y)\alpha_1]^6\beta
\end{eqnarray*}

viewed in $\xc_1$
$$\begin{array}{ccccrcr}
(A^7-7A^6B)_{dom} &= &s_5(\cf_1)&+& 7(2+y)\ \alpha_1\  s_4(\cf_1) &-& 7\ep x s_4(\cf_1)\beta\\
&&&+& 21(2+y)^2\alpha_1^2s_3(\cf_1) &-& 7\times 6\ep x (2+y) s_3(\cf_1)\alpha_1\beta\\
&&&+& 35(2+y)^3\alpha_1^3s_2(\cf_1) &-& 7\times 15\ep  x (2+y)^2 s_2(\cf_1)\alpha_1^2\beta\\
&&&+& 35(2+y)^4\alpha_1^4s_1(\cf_1) &-& 7\times 20\ep x (2+y)^3 s_1(\cf_1)\alpha_1^3\beta\\
&&&+& 21(2+y)^5 \alpha_1^5   \ \ \ \         &-& 7\times 15\ep x (2+y)^4\ \ \ \         \alpha_1^4 \beta
\end{array}
$$
This leads to the following expression for the dominant term, viewed in $\xc$
\begin{eqnarray*}
 (A^7-7A^6B)_{dom} &=& [-2-14(2+y)+63(2+y)^2-70(2+y)^3+35(2+y)^4]s_1s_2\\
&&-7\ep x  [-13+42(2+y)-45(2+y)^2+20(2+y)^3]s_1^2\beta\\
&=&[-2-14(2+y)+63(2+y)^2-70(2+y)^3+35(2+y)^4](\chi d^3-12rd^2)\\
&&-7\ep x  [-13+42(2+y)-45(2+y)^2+20(2+y)^3]d^3\\
&=&\left(222+518\,y+483\,{y}^{2}+210\,{y}^{3}+35\,{y}^{4}\right)(\chi d^3-12rd^2)\\
&&-7\ep x\left(51+102\,y+75\,{y}^{2}+20\,{y}^{3}\right)d^3
\end{eqnarray*}
We can apply Schwarz lemma provided $\ep x>\chi (3+2y)$.
This would lead to 
\begin{eqnarray*}
 (A^7-7A^6B)_{dom} &\leq&\left(222+518\,y+483\,{y}^{2}+210\,{y}^{3}+35\,{y}^{4}\right)(\chi d^3-12rd^2)\\
&&-7\chi (3+2y)\left(51+102\,y+75\,{y}^{2}+20\,{y}^{3}\right)d^3\\
 &\leq&-\left(849+2338\,y+2520\,y^2+1260\,y^3+245\,y^4\right) \chi d^3\\
&&-\left( 2664+6216\,y+5796\,{y}^{2}+2520\,{y}^{3}+420\,{y}^{4} \right) r{d}^{2}
\end{eqnarray*}

\subsection{On $\xc_3$}

Here we take $A$ and $B$ with first Chern class 
\begin{eqnarray*}
a
&=&(\alpha_3+2\alpha_2+6\alpha_1+18\alpha)+z(\alpha_2+2\alpha_1+6\alpha)\\
&&+y(\alpha_1+2\alpha) + x(\alpha-\ep \beta)\\
&=&\alpha_3+(2+z)\alpha_2+(6+2z+y)\alpha_1+(18+6z+2y+x)\alpha-\ep x\beta
\end{eqnarray*}
and $$b=(18+6z+2y+x)\alpha.$$

The bundle $A\otimes B^{-1}$ is $\oc_{\xc_2}(-\ep x, 0;6+2z+y,2+z,1)$. 
In order to apply Schwarz lemma, we choose 
$$\ep x=9+3z+y.$$

 The dominant term of $A^9-9A^8B$ is (computed with Maple)

$( 34272\,{y}^{3}z+3304896\,{z}^{3}+17136\,{z}^{6}+25200\,{y}^{2}
{z}^{4}+1332648+906336\,y
+3997944\,z+495936\,{y}^{2}z+34272\,y{z}^{5}+181440\,{y}^{2}{z}^{3}
+222768\,{z}^{5}+212544\,{y}^{2}
+2416896\,yz+1391040\,y{z}^{3}+1189440\,{z}^{4}+5016096\,{z}^{2}
+352800\,y{z}^{4}+17136\,{y}^{3}+25200\,{y}^{3}{z}^{2}
+6720\,{y}^{3}{z}^{3}+2613744\,y{z}^{2}+450576\,{y}^{2}{z}^{2} ) rd^3$

$
-( 869904\,{y}^{3}z+44108988\,{z}^{3}+559608\,{z}^{6}+32130\,{y}^{4}z
+772380\,{y}^{2}{z}^{4}+16542612+12428586\,y
+49627836\,z+8196300\,{y}^{2}z+18900\,{y}^{4}{z}^{2}
+1085616\,y{z}^{5}+3507840\,{y}^{2}{z}^{3}
+3780\,{y}^{4}{z}^{3}+4306554\,{z}^{5}+3512700\,{y}^{2}+30564\,{z}^{7}
+33142896\,yz+21170016\,y{z}^{3}+19278\,{y}^{4}
+18008802\,{z}^{4}+63329508\,{z}^{2}+6674220\,y{z}^{4}+434952\,{y}^{3}
+664020\,{y}^{3}{z}^{2}+221760\,{y}^{3}{z}^{3}
+26460\,{y}^{3}{z}^{4}+36642312\,y{z}^{2}+65016\,{y}^{2}{z}^{5}
+7663572\,{y}^{2}{z}^{2}+71316\,{z}^{6}y ) \chi d^3.
$

\bigskip

\enlargethispage{2cm}
\end{document}